\documentclass[11pt]{amsart}

\usepackage{amscd,amssymb,amsopn,amsmath,amsthm,mathrsfs,graphics,amsfonts,enumerate,verbatim,calc
}
\usepackage[all]{xypic}
\usepackage{url}
\usepackage{stmaryrd}
\usepackage{color}

\usepackage[OT2,OT1]{fontenc}
\newcommand\cyr{%
\renewcommand\rmdefault{wncyr}%
\renewcommand\sfdefault{wncyss}%
\renewcommand\encodingdefault{OT2}%
\normalfont
\selectfont}
\DeclareTextFontCommand{\textcyr}{\cyr}

\usepackage{amssymb,amsmath}

\DeclareFontFamily{OT1}{rsfs}{}
\DeclareFontShape{OT1}{rsfs}{n}{it}{<-> rsfs10}{}
\DeclareMathAlphabet{\mathscr}{OT1}{rsfs}{n}{it}

\numberwithin{equation}{section}
\newtheorem{theorem}{Theorem}[section]
\newtheorem{lemma}[theorem]{Lemma}
\newtheorem{proposition}[theorem]{Proposition}
\newtheorem{corollary}[theorem]{Corollary}

\newtheorem{maintheorema}{Main Theorem A}

\newtheorem{maintheoremb}{Main Theorem B}

\theoremstyle{definition}
\newtheorem{definition}[theorem]{Definition}
\newtheorem{remark}[theorem]{Remark}
\theoremstyle{remark}

\newtheorem{example}[theorem]{Example}
\newtheorem{acknowledgement}{Acknowledgement}

\newcommand{\im}{\operatorname{Im}}
\renewcommand{\ker}{\operatorname{Ker}}

\newcommand{\Spec}{\operatorname{Spec}}

\newcommand{\Ht}{\operatorname{ht}}

\newcommand{\Div}{\operatorname{Div}}
\newcommand{\gr}{\operatorname{gr}}

\newcommand{\Tor}{\operatorname{Tor}}

\newcommand{\Cl}{\operatorname{Cl}}

\newcommand{\ddiv}{\operatorname{div}}

\newcommand{\rank}{\operatorname{rank}}

\newcommand{\fm}{\frak{m}}
\newcommand{\fp}{\frak{p}}
\newcommand{\fq}{\frak{q}}

\newcommand{\fn}{\frak{n}}

\newcommand{\mbN}{\mathbb{N}}

\newcommand{\sharQ}{\mcQ_{\redu}}

\newcommand{\mcQ}{\mathcal{Q}}
\newcommand{\redu}{\operatorname{red}}
\newcommand{\Prin}{\operatorname{Prin}}
\newcommand{\gp}{\operatorname{gp}}
\newcommand{\relint}{\operatorname{relint}}

\newcommand{\bfi}{\textbf{I}}

\newcommand{\bfiii}{\textbf{III}}

\begin{document}
\title[Local log-regular rings vs. toric rings]
{Local log-regular rings vs. toric rings}

\author[S.Ishiro]{Shinnosuke Ishiro}
\address{National Institute of Technology, Gunma College, 580 Toriba-machi, Maebashi-shi, Gunma 371-8530, Japan}
\email{shinnosukeishiro@gmail.com}


\begin{abstract}
Local log-regular rings are a certain class of Cohen-Macaulay local rings that are treated in logarithmic geometry.
Our paper aims to provide purely commutative ring theoretic proof of several ring-theoretic properties of local log-regular rings such as an explicit description of a canonical module, and the finite generation of the divisor class group.
\end{abstract}

\maketitle

\section{Introduction}
In his paper \cite{Kat94}, Kato established the theory of toric geometry without a base by using logarithmic structures of Fontaine--Illusie.
He named the schemes appearing in the theory \textit{log-regular schemes}.
Their local rings are referred to as \textit{local log-regular rings}.
The class of local log-regular rings has properties similar to those of toric rings (for example, they are Cohen--Macaulay and normal).
Moreover, this class has a structure theorem analogous to Cohen's structure theorem (cf. Theorem \ref{Katostructure}).
By the structure theorem, the completion of a local log-regular ring can be expressed as a complete monoid algebra.

The class of local log-regular rings is also important from the perspective of commutative ring theory in mixed characteristic.
Gabber and Ramero explicitly constructed a perfectoid ring that is an algebra over a local log-regular ring (\cite[\S 17.2]{GRBook} or \cite[Construction 3.58]{INS22}).
They apply the construction to prove that a local log-regular ring is a splinter.\footnote{In a joint work with K. Nakazato and K. Shimomoto, we gave more elementary proof for this result using the Direct Summand Conjecture. See \cite[Theorem 2.27]{INS22}}
Moreover, Cai--Lee--Ma--Schwede--Tucker recently proved that a complete local log-regular ring is BCM-regular, which is a BCM-analogue of strong $F$-regularity (\cite[Proposition 5.3.5]{CLMST22}).

In this paper, we explore ring-theoretic properties of local log-regular rings, in particular canonical modules and divisor class groups.
Firstly, we explore canonical modules of local log-regular rings.
The existence of a dualizing complex of a log-regular scheme is already proved by Gabber and Ramero \cite[Theorem 12.5.42]{GRBook}.
We investigate the structure of canonical modules of local log-regular rings explicitly.

\begin{maintheorema}[Theorem \ref{canonicalmod}]\label{mainthm}
Let $(R, \mathcal{Q}, \alpha)$ be a local log-regular ring, where $\mathcal{Q}$ is finitely generated, cancellative, reduced, and root closed (by Remark \ref{submono}, we may assume that $\mathcal{Q} \subseteq \mathbb{N}^l$ for some $l>0$ ).
Let $x_1, \ldots, x_r$ be a sequence of elements of $R$ such that $\overline{x_1}, \ldots, \overline{x_r}$ is a regular system of parameters for $R/I_\alpha$.
Then $R$ admits a canonical module and
\begin{equation}\label{mainthmeq}
\langle (x_1\cdots x_r)\alpha(a)~|~ a \in \relint \mcQ \rangle
\end{equation}
is the canonical module of $R$, where $\relint \mcQ$ is the relative interior of $\mathcal{Q}$.
\end{maintheorema}
Though the proof of \cite[Theorem 12.5.42]{GRBook} is sheaf-theoretic, we show Main Theorem A by reducing it to the case of a semigroup ring.
Robinson proved the toric case in \cite{Rob22}.
We mention the relationship between his result and Main Theorem A in Remark \ref{Robinsonrem}.

As applications of Main Theorem A, we provide a criterion for the Gorenstein property of local log-regular rings (Corollary \ref{logGorenstein}).
Moreover, we provide a structure theorem of Gorenstein local log-regular rings with two-dimensional monoids (Proposition \ref{exampletwodim}).
Additionally, we prove that local log-regular rings are pseudo-rational (Proposition \ref{logpseudo}).

Secondly, we show that the divisor class group of a local log-regular ring is finitely generated.
To establish this, we prove that a log structure induces the isomorphism between the divisor class group of a local log-regular ring and that of the associated monoid.

\begin{maintheoremb}[{Theorem \ref{fgcl}}]
    Let $(R, \mathcal{Q}, \alpha)$ be a local log-regular ring.
Then $\alpha$ induces the the group homomorphism $\Cl(\alpha) : \Cl(\mathcal{Q}) \to \Cl(R)$ and it is an isomorphism.
In particular, $\Cl(R)$ is finitely generated.
\end{maintheoremb}

Combining Main Theorem B with Chouinard's result in \cite{Cho81}, we find that the divisor class group of a local log-regular ring $(R,\mcQ,\alpha)$ is isomorphic to that of the monoid algebra $k[\mcQ]$ over a field $k$.

At the end of the introduction, we provide an outline of this paper.
In \S \ref{SecPreMon}, we present the basic notions of monoids.
In \S \ref{Seclogreg}, we introduce the definition and certain properties of local log-regular rings.
In \S \ref{SectCanonical}, we prove Main Theorem A.
We also provide examples of Gorenstein local log-regular rings.
In \S \ref{SecDiv}, we prove Main Theorem B.

\section{Preliminaries on monoids}\label{SecPreMon}
\subsection{Properties on monoids}
A \textit{monoid} $\mathcal{Q}$ is a commutative semigroup with a unit.
We denote the group of units (resp. the group consisting of elements in the form of $q-p$ where $q,p \in \mcQ$) by $\mathcal{Q}^\times$ (resp. $\mathcal{Q}^{\gp}$).
We also denote by $\mathcal{Q}^{+}$ the set of non-unit elements of $\mathcal{Q}$ (i.e.\ $\mathcal{Q}^{+}=\mathcal{Q}\setminus \mathcal{Q}^\times$). 
Let us recall the terminologies of monoids.
\begin{definition}\label{DefProperties}
Let $\mathcal{Q}$ be a monoid.
\begin{enumerate}
\item
$\mathcal{Q}$ is called \textit{cancellative} if for $x, x'$ and $y \in \mcQ$, $x+y=x' +y$ implies $x=x'$.
\item
$\mathcal{Q}$ is called \textit{reduced} if $\mcQ^{\times}=0$.
\item
$\mcQ$ is called \textit{root closed} if it satisfies the following conditions.
    \begin{itemize}
        \item $\mcQ$ is cancellative.
        \item If $x \in \mcQ^{\gp}$ such that $nx \in \mcQ$ for some $n >0$, then $x \in \mcQ$.
        \end{itemize}
\end{enumerate}
\end{definition}

\begin{remark}
    In the context of logarithmic geometry, the terminologies defined in Definition \ref{DefProperties} are referred to by different names.
    For instance, in Ogus's book \cite{OgusBook}, "cancellative" is called "\textit{integral}", "reduced" is called "\textit{sharp}", and "root closed" is called "\textit{saturated}".
\end{remark}

\begin{definition}
    Let $\mcQ$ be a monoid.
    Then an equivalence relation $\sim$ on $\mcQ$ is called \textit{congruence} if $a \sim b$ implies $a+c \sim b+c$ for any $a,b,c \in \mcQ$.
\end{definition}

\begin{example}[Associated reduced monoids]
    Let $\mcQ$ be a monoid.
    Two elements $a, b \in \mcQ$ are called \textit{associates} if there exists a unit $u \in \mcQ^\times$ such that $a=u+b$.
    If $a, b \in \mcQ$ are associates, we denote them by $a \simeq b$.
    The relation $\simeq$ is a congruence relation and the monoid $\mcQ_{\redu} := \mcQ/\simeq$ is called the \textit{associated reduced monoid of $\mcQ$}.
    By definition, we have $[a] = a + \mcQ^\times$ where $[a]$ is an element of $\mcQ_{\redu}$.
    This implies that $\mcQ$ is reduced if and only if $\mcQ \cong \mcQ_{\redu}$.
\end{example}

We recall ideals and prime ideals of monoids.
\begin{definition}
Let $\mathcal{Q}$ be a monoid.
\begin{enumerate}
\item
A subset $I$ of $\mathcal{Q}$ is an \textit{$s$-ideal} if $a+x \in I$ for any $a \in \mathcal{Q}$ and any $x \in I$.
\item
An ideal $\fp \subseteq \mathcal{Q}$ is called \textit{prime} if $\fp \neq \mathcal{Q}$, and for $p, q \in \mathcal{Q}$, $p+q \in \mathcal{Q}$ implies $p \in \mathcal{Q}$ or $q \in \mathcal{Q}$.
\item
The set of primes ideals of $\mathcal{Q}$ is called the \textit{spectrum of $\mathcal{Q}$} and is denoted by $\Spec (\mathcal{Q})$.
\end{enumerate}
\end{definition}
The spectrum of a monoid becomes a topological space.
We note that the empty set $\emptyset$ and the set $\mathcal{Q}^+$ are prime ideals.
Moreover, $\emptyset$ is the unique minimal prime, and $\mathcal{Q}^+$ is the unique maximal ideal.
We define the dimension of a monoid.
\begin{definition}
The \textit{dimension} of a monoid $\mathcal{Q}$ is the maximum length $d$ of a chain of prime ideals
\begin{center}
$\emptyset=\fq_{0} \subsetneq \fq_{1} \subsetneq \cdots \subsetneq \fq_{d}=\mathcal{Q}^{+}$.
\end{center}
We denote it by $\dim (\mathcal{Q})$.
\end{definition}
Let $\varphi : \mathcal{Q} \to \mathcal{Q}'$ be a monoid homomorphism and let $\fp$ be a prime ideal of $\mathcal{Q}'$.
Then $\varphi^{-1}(\fp)$ is also prime.
Thus one can define the map $\Spec(\varphi) : \Spec (\mathcal{Q}') \to \Spec (\mathcal{Q})$.

\begin{proposition}\label{fundamentalmspc}
    Let $\mcQ$ be a finitely generated and cancellative monoid.
    Then $\Spec(Q)$ is a finite set.
\end{proposition}

\begin{proof}
This is \cite[Chapter \bfi, Propositoin 1.4.7 (1)]{OgusBook}.
\end{proof}

The following lemma follows from a discussion of convex polyhedral cones.
\begin{lemma}\label{transformation}
Let $\mathcal{Q}$ be a finitely generated, cancellative, and reduced monoid.
Then the following assertions hold.
\begin{enumerate}
\item\label{8159271}
The equality $\dim (\mcQ) =\rank(\mathcal{Q}^{\gp})$ holds.
\item\label{8159272}
Assume that $\mathcal{Q}^{\gp}$ is a torsion-free abelian group of rank $r$.
    Then there is an injective monoid homomorphism $\mathcal{Q} \hookrightarrow \mathbb{N}^r$.
\end{enumerate}
\end{lemma}

\begin{proof}
    The assertion (\ref{8159271}) is \cite[Corollary 6.4.12 (i)]{GRBook} and the assertion (\ref{8159272}) is \cite[Corollary 6.4.12 (iv)]{GRBook}.
\end{proof}


\begin{lemma}\label{decompmon}
    Let $\mcQ$ be a cancellative monoid such that $\sharQ$ is finitely generated and cancellative.
    Then there exists an isomorphism of monoids
    $$\mcQ \cong \sharQ \times \mcQ^\times. $$
\end{lemma}
\begin{proof}
    This is {\cite[Lemma 6.2.10]{GRBook}}.
\end{proof}

\subsection{Krull monoids and their divisor class groups}
In this subsection, we give a review of divisor class groups of Krull monoids.
Krull monoids have a long history in factorization theory and they are related to many mathematical fields, such as algebraic number theory, analytic number theory, combinatorial theory, and commutative ring theory.
For details, we refer the reader to \cite{BGBook}, \cite{GRBook}, \cite{GFBook}, or \cite{GZ20}.
First, we define fractional ideals of monoids.

\begin{definition}[Fractional ideals of monoids]
Let $\mathcal{Q}$ be a cancellative monoid.
Then a \textit{fractional ideal} of $\mathcal{Q}$ is a $\mathcal{Q}$-submodule $I \subseteq \mathcal{Q}^{\gp}$ such that $I \neq \emptyset$ and $x I := \{ x+a ~|~ a \in I \} \subseteq \mathcal{Q}$ for some $x \in \mathcal{Q}$.
\end{definition}

\begin{lemma}
Let $\mathcal{Q}$ be a cancellative monoid.
Then the following assertions hold.
\begin{enumerate}
\item
If $I_1, \ldots, I_n$ are fractional ideals of $\mcQ$, then $\displaystyle\bigcap_{i=1}^n I_i$ is also fractional.
\item
If $I_1, I_2$ are fractional ideals of $\mcQ$, then $I_1 I_2 := \{ x+y ~|~ x\in I_1, y \in I_2 \}$ is also fractional.
\end{enumerate}
\end{lemma}

\begin{proof}
(1):
Set $J:= \displaystyle\bigcap_{i=1}^n I_i$.
Since $I_i$ is a fractional ideal, there exists an element $a_i \in \mathcal{Q}$ such that $a_i I_i \subseteq \mathcal{Q}$.
Then $a_iJ \subseteq a_iI_i \subseteq \mathcal{Q}$.

(2):
Pick elements $a_1, a_2 \in \mcQ$ such that $a_1 I_1 \subseteq \mcQ$ and $a_2 I_2 \subseteq \mcQ$.
Then, since $a_1a_2(I_1 I_2) \subseteq \mcQ$, we can simply set $x=a_1a_2$.
\end{proof}

We say that a fractional ideal $I$ is \textit{finitely generated} if it is finitely generated as a $\mathcal{Q}$-module.
For any two fractional ideals $I_1$ and $I_2$, we define $(I_1: I_2) := \{ x \in \mathcal{Q}^{\gp}~|~ x I_2 \subseteq I_1 \}$.

\begin{lemma}
Let $\mathcal{Q}$ be a cancellative monoid, and let $I_1$ and $I_2$ be fractional ideals.
Then $(I_1: I_2)$ is also a fractional ideal.
\end{lemma}

\begin{proof}
Let $a_1 \in \mathcal{Q}^{\gp}$ such that $a_1I_1 \subseteq \mcQ$.
Pick an element $a \in I_2$.
For any $z \in (I_1 : I_2)$, $az \in I_1$.
Thus $a_1 az \in a_1 I_1 \subseteq \mcQ$.
This implies $a_1 a (I_1 : I_2) \subseteq \mcQ$, as desired.
\end{proof}

For a fractional ideal $I$ of $\mcQ$, we set $I^{-1} := (P:I)$ and $I^* := (I^{-1})^{-1}$.
We say that a fractional ideal $I$ is \textit{divisorial} (or \textit{$v$-ideal}) if $I^* = I$ holds.

\begin{lemma}
Let $\mathcal{Q}$ be a cancellative monoid, and let $I$ and $J$ be fractional ideals.
Then the following assertions hold.
\begin{enumerate}
\item
If $I \subseteq J$, then $J^{-1} \subseteq I^{-1}$ and $I^* \subseteq J^*$ hold.
\item
$I \subseteq I^*$ holds.
\item 
$I^*$ is divisorial.
Especially, $I^*$ is the smallest divisorial ideal containing $I$.
\item
For any $a \in \mathcal{Q}^{\gp}$, $aI^{-1} = (a^{-1}I)^{-1}$ and $aI^* = (aI)^*$ hold.
\item
$(IJ)^* = (I^*J^*)^*$ holds.

\end{enumerate}
\end{lemma}

\begin{proof}
(1):
Let $a \in \mcQ^{\gp}$ such that $a J \subseteq \mcQ$.
Since $I \subseteq J$, we have $aI \subseteq aJ \subseteq \mcQ$, as desired.
The latter assertion follows from the former assertion.

(2):
Pick $a \in I$.
For any $z \in (\mcQ: I)$, $zI \subseteq \mathcal{Q}$, in particular $za \in \mcQ$.
Thus $a \in (\mcQ : (\mcQ : I))$.

(3):
This is the same proof as in \cite[Lemma 1.2 (1)]{SamBook}.

(4):
The inclusion $aI^{-1} \subseteq (a^{-1}I)^{-1}$ obviously holds.
Conversely, pick an element $z \in (a^{-1}I)^{-1}$.
Then we have $(a^{-1}z) I = z (a^{-1}I) \subseteq \mcQ$.
This implies that $z \in aI^{-1}$, as desired.
Next, by the former equality, we obtain $aI^* = (a^{-1}I^{-1})^{-1} = (((a^{-1})^{-1}I)^{-1})^{-1} = (aI)^*$.

(5):
Pick $a \in I^*$.
Then $(aJ^*)^* = a(J^*)^* =aJ^*$.
This implies that $(I^* J^*)^* = I^*J^*$.
Pick $b \in J^*$.
Then $I^*b = (Ib)^*$.
This implies $I^*J^* = (IJ^*)^*$.
Finally, pick $c \in I$.
Then $(cJ^*)^* = (cJ)^{**} = (cJ)^*$.
This implies that $(IJ^*)^* = (IJ)^*$.
To summarize these, we obtain $(I^*J^*)^* = I^*J^* = (IJ^*)^* = (IJ)^*$, as desired.
\end{proof}

\begin{definition}
Let $\mathcal{Q}$ be a cancellative monoid.
We denote by $\Div(\mathcal{Q})$ the set of all divisorial ideals of $\mcQ$.
We define a binary operation on $\Div(\mathcal{Q})$ by
$$
I \bullet J := (IJ)^*.
$$
\end{definition}

Note that a monoid $\mcQ$ is a divisorial ideal.
Moreover, for a divisorial ideal $I$, we have $\mcQ \bullet I = I \bullet \mcQ = I$.
Hence $(\Div(\mcQ), \bullet)$ is a monoid.
In order to discuss the conditions under which $\Div(\mcQ)$ becomes a group, we define a monoid to be completely integrally closed.

\begin{definition}
Let $\mcQ$ be a cancellative monoid.
\begin{enumerate}
    \item 
    An element $x \in \mcQ^{\gp}$ is called \textit{almost integral over $\mcQ$} if there exists $c \in \mcQ$ such that $c+nx \in \mcQ$ for any $n\in\mathbb{Z}_{> 0}$.
    \item
    $\mcQ$ is called \textit{completely integrally closed} if all almost integral elements over $\mcQ$ lie in $\mcQ$.
\end{enumerate}
\end{definition}

The set of elements of $\mcQ^{\gp}$ which are almost integral over $\mcQ$ is a monoid.
Indeed, for almost integral elements $x, y \in \mcQ^{\gp}$, there exist elements $a,b \in \mcQ$ such that $a+nx, b+ny \in \mcQ$ for any $n \in \mathbb{Z}_{>0}$.
Since we have $(a+b) + n(x+y) = (a+nx)+(b+ny) \in \mcQ$, $x+y$ is also almost integral over $\mcQ$.

\begin{proposition}
Let $\mcQ$ be a cancellative monoid.
Then the following assertions hold.
\begin{enumerate}
\item
$(\Div(\mcQ), \bullet)$ is an abelian group if and only if $\mcQ$ is completely integrally closed.
\item
If $\mcQ$ is finitely generated, cancellative, and root closed, then $\mcQ$ is completely integrally closed.
\end{enumerate}
\end{proposition}

\begin{proof}
    These assertions are \cite[Proposition 6.4.42 (i), (ii)]{GRBook}.
\end{proof}


Next, we define Krull monoids.

\begin{definition}[Krull monoids]
Let $\mathcal{Q}$ be a cancellative monoid.
Then $\mcQ$ is a \textit{Krull monoid} if the following two conditions hold.
\begin{enumerate}
\item
The set of divisorial ideals of $\mcQ$ contained in $\mcQ$ satisfies the ascending chain condition, that is, for any sequence $I_0 \subseteq I_1 \subseteq I_2 \subseteq \cdots$ of divisorial ideals, there exists a number $n \geq 0$ such that $I_m = I_{m+1}$ for any $m \ge n$.

\item
$\mcQ$ is completely integrally closed.

\end{enumerate}
\end{definition}

\begin{lemma}\label{LemmaCIC}
    Let $\mcQ$ be a cancellative monoid.
    Then the following assertions hold.
    \begin{enumerate}
        \item 
        $\mcQ$ is completely integrally closed if and only if $\mcQ_{\redu}$ is completely integrally closed.
        \item
        $\mcQ$ is a Krull monoid if and only if $\mcQ_{\redu}$ is a Krull monoid.
    \end{enumerate}
\end{lemma}

\begin{proof}
    These are \cite[Corollary 2.3.6]{GFBook}.
\end{proof}

Krull monoids possess many properties similar to those of Krull rings.
In particular, Proposition \ref{LatticeDivprop} is important for providing concrete computations of divisor class groups.

\begin{proposition}\label{LatticeDivprop}
Let $\mcQ$ be a cancellative monoid and let $D \subseteq \Spec (\mcQ)$ be the subset of all prime ideals of height one.
Then $\mcQ$ is Krull if and only if there is an isomorphism $\mathbb{Z}^{\oplus D} \cong \Div (\mcQ)$ as an abelian group.
\end{proposition}

\begin{proof}
The proof is the same as in \cite[Theorem 3.1]{SamBook}.
\end{proof}

Krull monoids have many other characterizations. See \cite[Theorem 2.3.11]{GFBook} and \cite[Theorem 2.4.8]{GFBook} for details.

Keep the notation as in Proposition \ref{LatticeDivprop}.
Let us denote $(n_\fp)_{\fp \in D} \in \mathbb{Z}^{\oplus D}$ by $\sum_{\fp \in D} n_\fp \fp$.
Also let us denote $\ddiv : \mathbb{Z}^{\oplus D} \xrightarrow{\cong} \Div(\mcQ)$.

\begin{definition}
Let $\mcQ$ be a cancellative monoid and let $a \in \mcQ^{\gp}$ be an element.
Then we define a \textit{principal fractional ideal} as $\{ a+q ~|~q\in\mcQ$\}.
Moreover, we denote the set of principal fractional ideals by $\Prin (\mcQ)$.
\end{definition}

Let $\mcQ$ be a cancellative monoid and let $I, J$ be fractional ideals of $\mcQ$.
Here we define $I \sim J$ if there exists an element $a\in \mcQ^{\gp}$ such that $I =  a J$.
Then $\sim$ is an equivalence relation.

\begin{definition}[The divisor class groups of monoids]
Let $\mcQ$ be a cancellative monoid.
Then we define \textit{the divisor class group of $\mcQ$} as $\Div(\mcQ)/\sim $ and denote it by $\Cl(\mcQ)$.
\end{definition}

For a cancellative monoid $\mathcal{Q}$, $\Cl(\mcQ)$ is a monoid (its binary operation is induced by that of $\Div(\mcQ)$).
Furthermore, if $\mcQ$ is completely integrally closed, then $\Cl(\mcQ)$ is an abelian group.

Here assume that $\mcQ$ is a Krull monoid.
Let $\fp \in \Spec(Q)$ be a height one prime ideal of $\mcQ$.
If $\fp$ is a principal ideal, then $\ddiv(\fp)$ is contained in a principal fractional ideal of $\mcQ$ by Proposition  \ref{LatticeDivprop}.
Hence we obtain $\ddiv^{-1}(\Prin(Q)) = \{  \sum_{\Ht\fp = 1} n_\fp \fp \in \mathbb{Z}^D ~|~ \fp\text{ is principal} \}$ and 
\begin{equation}
\overline{\ddiv} : \mathbb{Z}^D/\ddiv^{-1}(\Prin(Q)) \xrightarrow{\cong} \Cl(\mcQ).
\end{equation}
By this isomorphism, we obtain the following result.

\begin{corollary}
    Let $\mcQ$ be a Krull monoid.
    Then the following assertions are equivalent.
    \begin{enumerate}
        \item $\Cl(\mcQ) = 0$.
        \item Any height one prime ideal of $\mcQ$ is principal.
    \end{enumerate}
\end{corollary}

\section{Local log-regular rings}\label{Seclogreg}

In this section, we provide an overview of the definition and fundamental properties of log-regularity of commutative rings.
First, we provide the log structure of commutative rings.
\begin{definition}
Let $R$ be a ring, let $\mcQ$ be a monoid, and let $\alpha : \mcQ \to R$ be a monoid homomorphism.
\begin{enumerate}
\item
The triple $(R, \mathcal{Q}, \alpha)$ is called a \textit{log ring}.
\item
A log ring $(R, \mathcal{Q}, \alpha)$ is called a \textit{local log ring} if $R$ is local and $\alpha^{-1}(R^\times) = \mathcal{Q}^\times$, where $R^\times$ is the group of units of $R$.
\end{enumerate}
\end{definition}
Here, we define log-regularity of commutative rings.
\begin{definition}[{cf. \cite[Chapter {\bf III}, Section 1.11]{OgusBook}}]\label{log-regular}
Let $(R, \mathcal{Q}, \alpha)$ be a local log ring, where $R$ is Noetherian, $\mathcal{Q}$ is cancellative, and $\sharQ$ is finitely generated and root closed.
Let $I_{\alpha}$ be the ideal of $R$ generated by $\alpha(\mcQ^+)$.
Then $(R, \mcQ, \alpha)$ is called a \textit{local log-regular ring} if it satisfies the following conditions.
\begin{enumerate}
\item
$R/I_{\alpha}$ is a regular local ring.
\item
The equality $\dim (R) = \dim (R/I_{\alpha}) + \dim (\mcQ)$ holds.
\end{enumerate}
\end{definition}

\begin{remark}\label{rmk2125}
    We note that a monoid $\mcQ$ appearing in Definition \ref{log-regular} has a decomposition $\mcQ \cong \sharQ \times \mcQ^\times$ by Lemma \ref{decompmon}. 
    This implies that the natural projection $\pi: \mcQ \twoheadrightarrow \sharQ$ splits as a monoid homomorphism, that is, $\alpha$ factors through $\pi$.
    Hence we obtain another log structure $(R, \sharQ, \alpha_{\redu})$ where $\alpha_{\redu} : \sharQ \to R$ is the monoid homomorphism such that $\alpha = \alpha_{\redu} \circ \pi$.
\end{remark} 



\begin{lemma}\label{2241637}
Let $(R,\mcQ, \alpha)$ be a log ring, where $\mcQ$ is a cancellative monoid.
Assume that $\alpha$ is injective.
Then the image of $\alpha$ is contained in $R^\bullet = R\backslash \{ 0\}$.
\end{lemma}

\begin{proof}
If $\mathcal{Q}$ is the zero monoid, the claim holds obviously.
Thus we may assume that $\mathcal{Q}$ is a non-zero monoid.
Suppose that there exists $x \in \mathcal{Q}$ such that $\alpha(x) = 0$.
Then, for a non-zero element $y \in \mcQ$, we have the equality $\alpha(x+y) = \alpha(x)$.
Since $\alpha$ is injective and $\mcQ$ is cancellative, we obtain $y=0$.
This is a contradiction.
Thus $\im \alpha \subseteq R^\bullet$ holds.
\end{proof}

In the situation of Lemma \ref{2241637}, the homomorphism $\alpha : \mathcal{Q} \to R$ factors through $\alpha^\bullet : \mathcal{Q} \to R^\bullet$.

\begin{definition}
A monoid homomorphism $\theta : \mathcal{P} \to \mathcal{Q}$ is \textit{exact} if the following diagram is cartesian:
\[
\xymatrix{
\mathcal{P}^{\gp} \ar[r] & \mathcal{Q}^{\gp} \\
\mathcal{P} \ar[u] \ar[r]& \mathcal{Q}, \ar[u]
}
\]
that is, $\mathcal{Q} \times_{\mathcal{Q}^{\gp}} \mathcal{P}^{\gp} = \mathcal{P}$.
\end{definition}

By Theorem \ref{Katostructure} which we will introduce later, we obtain the injectivity of the monoid homomorphism $\alpha$ of a local log-regular ring $(R, \mcQ, \alpha)$.
To provide a criterion of Gorensteinness of a local log-regular ring (Corollary \ref{logGorenstein}), we need the exactness of $\alpha^\bullet$.

\begin{lemma}
Let $(R, \mathcal{Q}, \alpha)$ be a local log-regular ring.
Assume that $\mathcal{Q}$ is finitely generated, cancellative, reduced, and root closed.
Then $\alpha^\bullet$ is exact.
\end{lemma}

\begin{proof}
Since $\mathcal{Q}$ is finitely generated, cancellative, and root closed and $R^\bullet$ is cancellative, it suffices to show that $\Spec (\alpha^\bullet)$ is surjective by \cite[Chapter {\bf I}, Proposition 4.2.2]{OgusBook}.
For any $\fq \in \Spec(\mcQ)$, $\fq R$ is prime of $R$ and $\alpha^{-1}(\fq R) = \fq$ by Proposition \ref{characterizationlogreg}.
Set $\fq^\bullet := \fq \backslash \{ 0 \} \subseteq R^\bullet$.
Since $\fq^\bullet$ is a prime ideal of $R^\bullet$, $\Spec (\alpha^\bullet) (\fq^\bullet) =\fq$ holds.
Hence $\Spec(\alpha^\bullet)$ is surjective.
\end{proof}


Let $\mcQ$ be a finitely generated, cancellative, reduced monoid, and let $R$ be a commutative ring.
Then we denote by $R\llbracket \mcQ \rrbracket$ the set of functions $\mcQ \to R$, viewed as an $R$-module using the usual point-wise structure and endowed with the product topology induced by the discrete topology on $R$, that is, we have the explicit description
 $$ R\llbracket \mcQ \rrbracket = \Bigl\{ \sum_{q \in \mcQ} a_qe^q~|~ a_q \in R\Bigr\}.$$
By using this description, the $R$-module $R\llbracket \mcQ \rrbracket$ admits the unique multiplication (see \cite[Chapter \bfi, Proposition 3.6.1 (2)]{OgusBook}).
Also, $R\llbracket \mcQ \rrbracket$ of $\mcQ$ can be view as the completion of $R[\mcQ]$ with respect to an ideal $R[\mcQ^+]$ (see \cite[Chapter \bfi, Proposition 3.6.1 (3)]{OgusBook}).


\begin{proposition}
Keep the notation as above.
Then the following assertions hold.
\begin{enumerate}
    \item If $\mcQ^{\gp}$ is torsion free and $R$ is also an integral domain, then $R\llbracket \mcQ\rrbracket$ is an integral domain.
    \item 
    If $R$ is a local ring with the maximal ideal $\fm$, then $R\llbracket \mcQ \rrbracket$ is a local ring with the maximal ideal consisting of elements of $R\llbracket\mcQ\rrbracket$ such that their constant term belongs to $\fm$.
\end{enumerate}
\end{proposition}

\begin{proof}
    These are \cite[Chapter \bfi, Proposition 3.6.1 (4)]{OgusBook} and \cite[Chapter \bfi, Proposition 3.6.1 (5)]{OgusBook}.
\end{proof}

The following theorem is an analogue of Cohen's structure theorem and it is one of the main tools to find out properties of local log-regular rings.

\begin{theorem}[{\cite[Chapter {\bf III}, Theorem 1.11.2]{OgusBook}} or {\cite[Theorem 2.20]{INS22}}]\label{Katostructure}
Let $(R,\mcQ, \alpha)$ be a local log ring, where $R$ is Noetherian and $\mathcal{Q}$ is finitely generated, cancellative, reduced, and root closed.
Let $k$ be the residue field of $R$.
Then the following assertions hold.
\begin{enumerate}
\item\label{Katostructure1}
Suppose that $R$ is of equal characteristic.
Then $(R, \mcQ, \alpha)$ is log-regular if and only if there exists a commutative diagram of the form
\begin{equation}\label{Cohendiag1}
\vcenter{
\xymatrix{
\mcQ \ar@{^(->}[r] \ar[d]_{\alpha} & k\llbracket \mcQ \oplus \mbN^r \rrbracket \ar[d]^{\phi}_{\cong} \\
R \ar[r] & \widehat{R},
}}
\end{equation}
where the top arrow is the natural injection and $\widehat{R}$ is the $\fm$-adic completion of $R$.

\item\label{Katostructure2}
Suppose that $R$ is of mixed characteristic. 
Let $C(k)$ be a Cohen ring of $k$ and let $p>0$ be the characteristic of $k$. 
Then $(R,\mcQ, \alpha)$ is log-regular if and only if there exists a commutative diagram
\begin{equation}\label{Cohendiag2}
\vcenter{
\xymatrix{
\mcQ \ar@{^(->}[r] \ar[d]_{\alpha} & C(k)\llbracket \mathcal{Q} \oplus \mathbb{N}^r\rrbracket \ar@{->>}[d]^{\phi} \\
R \ar[r] & \widehat{R},
}}
\end{equation}
where $\phi$ is a surjection and $\ker\phi$ is a principal ideal generated by an element $\theta \in C(k)\llbracket \mathcal{Q} \oplus \mathbb{N}^r\rrbracket$ whose constant term is $p$.
\end{enumerate}
Moreover, let ${\bf e}_1, \ldots {\bf e}_r$ be the canonical basis on $\mbN^r$ and let $x_1,\ldots,x_r$ be a sequence of elements of $R$ such that $\overline{x_1}, \ldots , \overline{x_r}$ is a regular system of parameters for $R/I_\alpha$. 
If $(R, \mathcal{Q}, \alpha)$ is a local log-regular ring, then one may assume that $\phi$ sends ${\bf e}_i$ to $\widehat{x_i}$ where $\widehat{x_i}$ is the image of $x_i$ in $\widehat{R}$.
\end{theorem}

\begin{proof}
The former assertions (1) and (2) are exactly \cite[Chapter {\bf III}, Theorem 1.11.2]{OgusBook}.
The latter assertion is obtained in the proof of \cite[Chapter {\bf III}, Theorem 1.11.2]{OgusBook}.
\end{proof}

\begin{definition}
Let $(R, \mathcal{Q}, \alpha)$ be a log ring.
\begin{enumerate}
    \item $R$ is \textit{$\alpha$-flat} if $\Tor^{\mathbb{Z}[\mathcal{Q}]}_1 (\mathbb{Z}[\mcQ]/\mathbb{Z}[I], R)=0$ 
 for any ideal $I \subseteq \mcQ$.
    \item $R$ is \textit{faithfully $\alpha$-flat} if $R$ is $\alpha$-flat and it satisfies the following condition:
    For a $\mathbb{Z}[\mathcal{Q}]$-module $M$, $R\otimes_{\mathbb{Z}[\mcQ]} M =0$ if and only if $M=0$.
\end{enumerate}
\end{definition}




Under the first condition in Definition \ref{log-regular}, the second condition is equivalent to several conditions.

\begin{proposition}\label{characterizationlogreg}
Keep the notation and the assumption as in Definition \ref{log-regular}.
Assume that $R/I_\alpha$ is regular.
Then the following conditions are equivalent:
\begin{enumerate}
    \item\label{charlogreg} $(R, \mathcal{Q},\alpha)$ is a local log-regular ring.
    \item\label{charverysolid} For every prime ideal $\fq$ of $\mcQ$, the ideal $\fq R$ generated by $\alpha(\fq)$ is a prime ideal of $R$ such that $\alpha^{-1}(\fq R) = \fq$.\footnote{In \cite{OgusBook}, the monoid homomorphism $\alpha$ of a log ring $(R,\mathcal{Q},\alpha)$ which satisfies the latter condition is called \textit{very solid}.}
    \item\label{charflat} $R$ is $\alpha$-flat.
    \item\label{chartor} $\Tor_1^{\mathbb{Z}[\mcQ]} (\mathbb{Z}[\mcQ]/\mathbb{Z}[\mcQ^+], R) =0$.
    \item\label{chargr} $\gr_{\mathbb{Z}[\mcQ^+]}(\mathbb{Z}[\mcQ]) \otimes_{\mathbb{Z}} R/I_{\alpha} \cong \gr_{I_{\alpha}} R$ is an isomorphism.
\end{enumerate}
\end{proposition}

\begin{proof}
    The equivalences $(1) \Leftrightarrow (2) \Leftrightarrow (4) \Leftrightarrow (5)$ are a combination of \cite[Chapter \bfiii, Theorem 1.11.1]{OgusBook} and \cite[Chapter \bfiii, Proposition 1.11.5]{OgusBook}.
    The equivalence $(1) \Leftrightarrow (3)$ is \cite[Proposition 52]{Tho06}.
\end{proof}

Conde-Lago and Majadas characterize local log-regular rings based on the vanishing of the homology of the logarithmic cotangent complex.
We refer the reader to \cite{CM22}.

We provide an example of non-complete local log-regular rings which is called a \textit{Jungian domain}.
This is defined by Abhyankar \cite{Abh65} (see also \cite[\S 12]{Kat94}).
He introduced it and explored how to construct it.
For example, see \cite[Theorem 10]{Abh65} or \cite[Theorem 14]{Abh65}.
Here we recall the definition of Jungian domains and provide an induced log structure.

\begin{definition}[{\cite[P23, Definition 2]{Abh65}}]\label{DefJungiandom}
    Let $(R,\fm)$ be a Noetherian local domain.
    We say that $(R, \fm)$ is a \textit{Jungian domain} if it is a two-dimensional normal domain such that the following condition is satisfied:
    There exist integers $m,n \in \mathbb{Z}$ with $0 \le m \le n$ and $\textnormal{GCD}(m,n)=1$, and generators $x, y, z_1, \ldots, z_{n-1}$ of $\fm$ such that $z_i^n = x^iy^{m_i}$ for any $i=1,\ldots, n-1$, where $m_i$ is the unique integer such that $0 \le m_i \le n$ and $m_i = mi ~(\textnormal{mod}~n)$.
\end{definition}

\begin{lemma}\label{Junglogreg}
    Let $(R,\fm)$ be a Jungian domain, let $\mathcal{M}$ be the multiplicative submonoid 
    $$\langle x^{l_1} y^{l_2} z_1^{l_{3}} \cdots z_{n-1}^{l_{n+1}} \in R~|~l_1,\ldots, l_{n+1} \geq 0  \rangle,$$
    and let $\alpha : \mathcal{M} \hookrightarrow R$ be the inclusion map.
    Then $\mathcal{M}$ is finitely generated, cancellative, reduced, and root closed.
    Moreover, $(R, \mathcal{M}, \alpha)$ is a local log-regular ring.
\end{lemma}

\begin{proof}
    Since $\mathcal{M}$ is generated by $x,y,z_1,\ldots, z_{i-1}$ and $R$ is a domain, $\mathcal{M}$ is finitely generated and cancellative.
    Pick an element $x \in \mathcal{M}^{\gp}$ such that $x^n \in \mathcal{M}$.
    Since $R$ is normal, we obtain $x \in R$.
    This implies that $x \in \mathcal{M}$ because we can show $\mathcal{M}^{\gp} \cap R = \mathcal{M}$.
    Hence $\mathcal{M}$ is root closed.
    Moreover, it follows from $I_\alpha = \fm$ that $\mathcal{M}$ is reduced and $R/I_\alpha$ is regular.
    Finally, we can easily check that any prime ideal of $\mathcal{M}$ forms $\fp \cap \mathcal{M}$ where $\fp$ is a prime ideal of $R$.
    Hence $\dim (\mathcal{M}) = \dim (R)$.
\end{proof}

It is well-known that a normal affine semigroup ring is Cohen--Macaulay and normal, which is proved by Hochster.
The same assertion holds for a local log-regular ring.

\begin{theorem}
Let $(R, \mathcal{Q}, \alpha)$ be a local log-regular ring.
Then 
$R$ is Cohen-Macaulay and normal.
\end{theorem}

\begin{proof}
See \cite[(4.1) Theorem]{Kat94} or \cite[Corollary 12.5.13]{GRBook}.
\end{proof}

\begin{remark}\label{submono}
If $\mathcal{Q}$ is finitely generated, cancellative, reduced, and root closed, then there is an exact injection $\mathcal{Q} \hookrightarrow \mathbb{N}^l$ for some $l \in \mathbb{N}$ (see \cite[Chapter {\bf I}, Proposition 1.3.5]{OgusBook} and \cite[Chapter {\bf I}, Corollary 2.2.7]{OgusBook}).
Thus, in the following sections, we assume that a finitely generated, cancellative, reduced, and root closed monoid is a submodule of some $\mathbb{N}^l$.
\end{remark}

\section{Canonical modules of local log-regular rings}\label{SectCanonical}
In this section, we provide an explicit structure of the canonical module of a local log-regular ring.

\begin{theorem}\label{canonicalmod}
Let $(R, \mathcal{Q}, \alpha)$ be a local log-regular ring, where $\mathcal{Q}$ is finitely generated, cancellative, reduced, and root closed (by Remark \ref{submono}, we may assume that $\mathcal{Q} \subseteq \mathbb{N}^l$ for some $l>0$ ).
Let $x_1, \ldots, x_r$ be a sequence of elements of $R$ such that $\overline{x_1}, \ldots, \overline{x_r}$ is a regular system of parameters for $R/I_\mathcal{Q}$.
Then $R$ admits a canonical module and its form is
\begin{equation}\label{canonicallogreg}
\langle (x_1\cdots x_r)\alpha(a)~|~ a \in \relint \mcQ \rangle
\end{equation}
where $\relint \mcQ$ is the relative interior of $\mathcal{Q}$.

\end{theorem}
\begin{proof}
First, assume that $R$ is $\fm$-adically complete and separated.
If $R$ is of equal characteristic, then $R$ is isomorphic to $k\llbracket \mcQ \oplus \mbN^r\rrbracket$ by Theorem \ref{Katostructure}.
Let us check that
\begin{equation}\label{canonicalsemigroup}
k[\relint\mcQ \oplus ({\bf e}+\mathbb{N}^r)] := 
 \langle (q, {\bf e}) ~|~ q \in \relint\mcQ \rangle \subseteq k[\mathcal{Q} \oplus \mathbb{N}^r]
\end{equation}
is a canonical module of $k[\mathcal{Q} \oplus \mbN^r]$, where $\textbf{e}=  (1,1,\ldots,1) \in \mathbb{N}^r$.
Indeed, note that we have the ring isomorphism
$k[\mcQ] \otimes_k k[\mbN^r] \cong k[ \mathcal{Q} \oplus \mbN^r]$.
Also note that the canonical module $\omega_{k[\mcQ]} = k[\relint\mathcal{Q}]$ and $\omega_{k[\mathbb{N}^r]} = k[{\bf e}+\mbN^r]$ by \cite[Theorem 6.3.5 (b)]{BH98}.
This induces the following isomorphisms
\begin{equation}\label{4241110}
\omega_{k[\mathcal{Q} \oplus \mathbb{N}^r]} \cong 
\omega_{k[\mcQ]}\otimes_k \omega_{k[\mbN^r]} =
k[\relint\mathcal{Q}] \otimes_k k[{\bf e} + \mbN^r]. 
\end{equation}
If you trace (\ref{4241110}) backwards, then it turns out that $\omega_{k[\mcQ \oplus \mbN^r]}$ is of the form of (\ref{canonicalsemigroup}).
Since $R$ is isomorphic to the completion of $k[\mathcal{Q}\oplus \mathbb{N}^r]$ along a maximal ideal $k[({\mathcal{Q}\oplus \mathbb{N}^r})^+]$, the image of $(\ref{canonicalsemigroup})$ in $R$ is the canonical module of $R$.

If $R$ is of mixed characteristic, then $R$ is isomorphic to $C(k)\llbracket \mathcal{Q} \oplus \mathbb{N}^r\rrbracket /(\theta)$ for some $\theta \in W(k)\llbracket \mathcal{Q} \oplus \mathbb{N}^r\rrbracket$.
If $C(k)\llbracket \mathcal{Q} \oplus \mathbb{N}^r\rrbracket$ has a canonical module, then its image in $R$ is the canonical module of $R$.
Thus it suffices to show the case where $R = C(k)\llbracket \mathcal{Q}\oplus \mathbb{N}^r\rrbracket$.

Set $\omega_R :=   \langle p (q, {\bf e}) ~|~ q \in \relint \mathcal{Q}    \rangle  \subseteq C(k)\llbracket \mcQ \oplus \mathbb{N}^r \rrbracket$.
Note that $\omega_R/p\omega_R$ is a canonical module of $R/pR \cong k\llbracket\mathcal{Q}\oplus \mathbb{N}^r\rrbracket$ and $p$ is a regular element on $R$ and $\omega_R$. Then $\omega_R$ is a maximal Cohen--Macaulay module of type $1$.
Finally, since $R$ is a domain, $\omega_R$ is faithful.
Thus $\omega_R$ is a canonical module of $R$.

Next, let us consider the general case.
We define the ideal $\omega_R$ as (\ref{canonicallogreg}).
Then, by considering the diagrams (\ref{Cohendiag1}) or (\ref{Cohendiag2}), the image of $\omega_R$ in the $\fm$-adic completion of $R$ is the canonical module.
Thus, by \cite[Theorem 3.3.14 (b)]{BH98}, $\omega_R$ is a canonical module of $R$.
\end{proof}

\begin{remark}\label{takagirem}
Set $\omega_R := \langle (x_1\cdots x_r)\alpha(a)~|~ a \in \relint \mcQ \rangle$ and $\omega_R' :=\langle\alpha(a)~|~ a \in \relint \mcQ \rangle$.
Then we note that the homomorphism $\omega_R' \xrightarrow{\times x_1\cdots x_r} \omega_R$ is an isomorphism.
Namely, the ideal of $R$ generated by the image of the relative interior of the associated monoid is also the canonical module of $R$.
\end{remark}

\begin{remark}\label{Robinsonrem}
    In Theorem \ref{canonicalmod}, the case when $R=W(k)\llbracket \sigma^{\lor} \cap M \rrbracket$ follows from the following Robinson's result\footnote{For readers who are not familiar with algebraic geometry, see \cite[Appendix B]{ST12}.}: Set $A:=W(k)[\sigma^{\lor} \cap M]$, where $M$ is a lattice and $\sigma$ is the strongly convex polyhedral cone.
    Set $X = \Spec(A)$. 
    Then one can choose codimension one subschemes $D_1, \ldots, D_n$ of $X$ such that $K_X = -\sum D_i$ is a canonical divisor on $X$.
    Indeed, his result implies that the ideal $\omega_A := \bigcap\fp_i$ is a canonical module of $A$, where $\fp_i$ is the corresponding height one prime ideal to $D_i$.
    By taking the localization and the completion at the maximal ideal $W(k)[(\sigma^\lor \cap M)^+]$, we find that $\omega_A$ is the canonical module of $W(k)\llbracket \sigma^{\lor} \cap M \rrbracket$.
    \end{remark}

As an application of Theorem \ref{canonicalmod}, let us provide a Gorenstein criterion of local log-regular rings.
In order to prove it, we need the following proposition.

\begin{proposition}\label{logregseq}
Let $(R,\mathcal{Q},\alpha)$ be a local log-regular ring.
Let $\underline{x} := x_1, \ldots, x_r$ be a sequence of elements of $R$ such that $\overline{x_1}, \ldots, \overline{x_r}$ is a regular system of parameters for $R/I_{\alpha}$.
Set $R_i:=R/(x_1,\ldots x_i )$ and $\alpha_i : \mathcal{Q} \to R \twoheadrightarrow R_i$.
Then $\underline{x}$ is a regular sequence on $R$ and $(R_i, \mathcal{Q}, \alpha_i)$ is also a local log-regular ring for any $1 \leq i \leq r$.
\end{proposition}

\begin{proof}
Since a local homomorphism preserves the locality of the log structure (see \cite[Lemma 2.16]{INS22}), $(R_i, \mathcal{Q}, \alpha_i)$ is a local log ring.
By the induction for $i$, it suffices to check the case $i=1$.
Since $R$ is a domain, $x_1$ is a regular element.
Thus we obtain the isomorphism $R_1/I_{\alpha_1} \cong (R/I_{\alpha})/x_1(R/I_{\alpha})$.
Since the image of $x_1$ is a regular element on $R/I_{\alpha}$ by the assumption and $R/I_{\alpha}$ is a regular local ring, $R_1/I_{\alpha_1}$ is regular.
Moreover, the above isomorphism implies that the equality
$\dim (R_1/I_{\alpha_1}) = \dim (R_1) -\dim(\mathcal{Q})$
holds.
Thus $(R_1, \mathcal{Q}, \alpha_1)$ is a local log-regular ring.
\end{proof}

\begin{corollary}\label{logGorenstein}
Keep the notation as in Theorem \ref{canonicalmod}.
The following assertions are equivalent.
\begin{enumerate}
\item
$R$ is Gorenstein.
\item
For a fixed field $k$, $k[\mathcal{Q}]$ is Gorenstein.
\item
There exists an element $c \in \relint Q$ such that $\relint \mathcal{Q} = c+\mathcal{Q}$.
\end{enumerate}
\end{corollary}

\begin{proof}
The equivalence of (2) and (3) is well-known (for example, see \cite[Theorem 6.3.5 (a)]{BH98}). 
Thus it suffices to show the equivalence of (1) and (3).
Since the Gorenstein property of $R$ is preserved under the completion and the quotient by a regular sequence, one can assume that $\alpha$ is injective by Theorem \ref{Katostructure} and that $\dim (R) = \dim (\mcQ)$ by Proposition \ref{logregseq}.
Hence $\omega_R = \langle \alpha(x)~|~ x \in \relint\mathcal{Q} \rangle$.
Now, assume that $R$ is Gorenstein.
There exists an element $c \in \relint \mcQ$ such that $\omega_R = \langle \alpha(c) \rangle$.
This implies that for any $a \in \relint \mcQ$, there exists $x \in R$ such that $\alpha(a) = \alpha(c)x$.
Since we have $\alpha(a) = \alpha^\bullet(a)$ and $\alpha(c) = \alpha^\bullet(c)$ by Lemma \ref{2241637}, we obtain 
\begin{equation}\label{2241639}
\alpha^\bullet(a) = \alpha^\bullet(c)x.
\end{equation}
Hence $x = \alpha^\bullet(a-c) \in \im\bigl(({\alpha^\bullet})^{\gp}\bigr)$.
Since $\alpha^\bullet$ is exact, we obtain $x \in \im\alpha^\bullet$.
Now, there exists $y \in \mathcal{Q}$ such that $x=\alpha^\bullet(y)$.
By (\ref{2241639}) and the injectivity of $\alpha^\bullet$, we obtain $a =c+y\in c+\mcQ$.
Hence $\relint \mathcal{Q} \subseteq c+\mcQ$.
Since $\relint \mcQ$ is an ideal of $\mathcal{Q}$, the converse inclusion holds.
Therefore we obtain $\relint \mcQ = c+\mathcal{Q}$.

Conversely, assume that $\relint \mcQ=c+\mathcal{Q}$ for some $c \in \relint \mcQ$.
Then we obtain the equalities $\omega_R = \alpha(c) \langle \alpha(x) \in R~|~ x \in \mcQ \rangle =\langle \alpha(c)\rangle$.
This implies that $R$ is Gorenstein, as desired.
\end{proof}

If a Cohen--Macaulay local ring has a canonical module, it is a homomorphic image of a Gorenstein local ring.
Namely, we obtain the following corollary.

\begin{corollary}\label{homomoim}
Let $(R,\mathcal{Q}, \alpha)$ be a local log-regular ring.
Then $R$ is a homomorphic image of a Gorenstein local ring.
\end{corollary}

Next, we prove that local log-regular rings are pseudo-rational. 
See \cite{LT81} or \cite{HM18} for the definition of pseudo-rationality.
The following theorem in equal characteristic is called Boutot's theorem.
It is proved in \cite{Bou87} and \cite{HH90}, and the analogue in mixed characteristic is proved in \cite{HM18}.

\begin{theorem}[{\cite{Bou87}, \cite{HH90}, \cite{HM18}}]\label{pseudorat}
 Let $(R, \fm) \to (S, \fn)$ be a pure map of local rings such that $(S,\fn)$ is regular.
Then $R$ is pseudo-rational.
In particular, direct summands of regular rings are pseudo-rational.
\end{theorem}

Applying this theorem, we obtain the pseudo-rationality of a local log-regular ring.

\begin{proposition}\label{logpseudo}
Let $(R,\mathcal{Q}, \alpha)$ be a local log-regular ring.
Then $R$ is pseudo-rational.
\end{proposition}

\begin{proof}
Since a local log-regular ring has the canonical module by Theorem \ref{canonicalmod}, by applying \cite[Proposition 4.20]{Mur22}, it suffices to show that a complete local log-regular ring is pseudo-rational.
Thus we may assume that $R$ is $\fm$-adically complete and separated.
Namely, $R$ is isomorphic to either 
$k\llbracket \mathcal{Q} \oplus \mathbb{N}^r \rrbracket$ or $C(k)\llbracket \mathcal{Q} \oplus \mathbb{N}^r\rrbracket/(\theta)$.
Now, we prove that $R$ is the direct summand of a regular local ring.
Our approach is the same as in the proof of \cite[Theorem 2.27]{INS22}, so we give the sketch of the proof here.
We refer the reader to it for details.
Since the same argument is made, we will show the case $R \cong C(k)\llbracket \mathcal{Q} \oplus \mathbb{N}^r\rrbracket/(\theta)$.
An embedding $\mathcal{Q} \hookrightarrow \mathbb{N}^r$ given in Remark \ref{submono} induces a split injection $C(k)[ \mathcal{Q} \oplus \mathbb{N}^r ] \hookrightarrow C(k)[\mathbb{N}^d]$ for some $d >0$.
This induces the split injection $C(k)\llbracket \mathcal{Q} \oplus \mathbb{N}^r \rrbracket \hookrightarrow C(k)\llbracket \mathbb{N}^d \rrbracket$.
By taking the quotient by some element $\theta \in A\llbracket \mathcal{Q} \oplus \mathbb{N}^r \rrbracket$, we also obtain the split injection $C(k)\llbracket \mathcal{Q} \oplus \mathbb{N}^r \rrbracket/(\theta) \hookrightarrow C(k)\llbracket \mathbb{N}^d \rrbracket/(\theta)$.
Finally, applying Theorem \ref{pseudorat}, we obtain the desired claim.
\end{proof}

\begin{remark}
    There is another way to prove Proposition \ref{logpseudo} in  the equal characteristic cases.
    If $R$ is an $F$-finite complete local log-regular ring, then it is a strongly $F$-regular ring.
    Since strong $F$-regularity implies $F$-rationality, $R$ is $F$-rational.
    Hence we obtain $R$ is pseudo-rational because an $F$-rational ring is pseudo-rational by \cite[Theorem 3.1]{Smi97}.
    Also, the equal characteristic $0$ case is due to \cite[Main Theorem A]{Schou08} and the above discussion.
\end{remark}

At the last of this section, we determine the form of Gorenstein local log-regular rings consisting of two-dimensional monoids by using Corollary \ref{logGorenstein}.

\begin{proposition}\label{exampletwodim}
Let $(R, \mcQ, \alpha )$ be a local log-regular ring where $\mcQ$ is finitely generated, cancellative, reduced, and root closed.
Assume that $\mcQ$ is two-dimensional.
Then $R$ is Gorenstein if and only if $\mathcal{Q}$ is isomorphic to the submonoid of $\mathbb{N}^2$ generated by $(n+1,0), (1,1), (0, n+1)$ for some $n \geq 1$.
\end{proposition}

\begin{proof}
By Corollary \ref{logGorenstein}, one can reduce to the case of a toric ring $k[\mathcal{Q}]$ where $k$ is an algebraically closed field, and in this case, we know that there exists $n \geq 1$ such that $k[\mathcal{Q}]$ is isomorphic to $k[\mathcal{P}]$ where $\mathcal{P}$ is the submonoid of $\mathbb{N}^2$ generated by $(n+1,0), (1,1), (0, n+1)$.
By applying \cite[Theorem 2.1 (b)]{Gub98} (see also \cite{BGBook09}), we can show that $\mathcal{Q}$ is isomorphic to $\mathcal{P}$, as desired.
Conversely, assume that $\mathcal{Q}$ is isomorphic to the submonoid generated by $(n+1,0), (1,1), (0, n+1) \in \mathbb{N}^2$.
Then $k[\mathcal{Q}]$ is Gorenstein for an algebraically closed field $k$ because this is an $A_n$-type singularity.
Thus $R$ is also Gorenstein by Corollary \ref{logGorenstein}, as desired.
\end{proof}

From the above proposition, it follows that a complete Gorenstein local log-regular ring with the two-dimensional monoid has the following form.

\begin{corollary}\label{strtwodim}
Let $(R, \mcQ, \alpha )$ be a local log-regular ring where $\mcQ$ is a two-dimensional finitely generated, cancellative, reduced, and root closed monoid.
Then the following assertions hold.
\begin{enumerate}
\item
Suppose that $R$ is of equal characteristic.
Then $R$ is Gorenstein if and only if $\widehat{R}$ is isomorphic to $k\llbracket s^{n+1}, st, t^{n+1}, x_1,\ldots, x_r \rrbracket$ for some $n\geq 1$.
\item
Suppose that $R$ is of mixed characteristic.
Then $R$ is Gorenstein if and only if $\widehat{R}$ is isomorphic to $C(k)\llbracket s^{n+1}, st, t^{n+1}, x_1,\ldots, x_r\rrbracket/(\theta)$ for some $n \geq 2$ where $C(k)$ is the Cohen ring of the residue field $k$ and $\theta$ is an element of $C(k)\llbracket s^{n+1}, st, t^{n+1}, x_1,\ldots, x_r\rrbracket$ whose constant term is $p$.

\end{enumerate}
\end{corollary}

\begin{proof}
These follow from Proposition \ref{exampletwodim} and Theorem \ref{Katostructure}.
\end{proof}

We also provide examples of non-Gorenstein local log-regular rings.
\begin{example}
Let $\mathcal{P}$ be the submonoid of $\mathbb{N}^2$ generated by
$
a_1 := (1,0), a_2 := (1,1), a_3 := (1,2), a_4 := (1,3).
$
\begin{enumerate}
\item
Set $R := \mathbb{Z}_p\llbracket \mathcal{P}\rrbracket / (p- e^{a_1}) \cong \mathbb{Z}_p\llbracket s, st, st^2,st^3 \rrbracket/(p-s)$ and set $\alpha : \mathcal{P} \to \mathbb{Z}_p\llbracket \mathcal{P}\rrbracket \to R$
Then $(R, \mathcal{Q}, \alpha)$ is a local log-regular ring.
By Proposition \ref{exampletwodim}, we know that $R$ is not Gorenstein.
Moreover, $R$ is also isomorphic to $\mathbb{Z}_p\llbracket pt, pt^2, pt^3 \rrbracket$.
Since the relative interior of $\mathcal{Q}$ is generated by $(1,1)$ and $(1,2)$, the canonical module of $R$ is generated by $e^{a_2}, e^{a_3}$ and it is isomorphic to the ideal of $\mathbb{Z}_p\llbracket pt, pt^2,pt^3 \rrbracket$ generated by $pt$ and $pt^2$.
\item
Set $S := \mathbb{Z}_p\llbracket \mathcal{P}\rrbracket$.
Then $(S, \mathcal{P}, \mathcal{P} \hookrightarrow S)$ is a local log-regular ring.
Then the $S$ is isomorphic to $\mathbb{Z}_p\llbracket s, st, st^2, st^3,x\rrbracket/(p-x)$. 
Then the canonical module of $S$ is isomorphic to the ideal generated by $pe^{a_2}$ and $pe^{a_3}$.
\end{enumerate}
\end{example}

\section{Divisor class groups of local log-regular rings}\label{SecDiv}

In this section, we prove that the divisor class group of a local log-regular ring is finitely generated.

\begin{lemma}\label{7231801}
Let $(R, \mathcal{Q}, \alpha)$ be a local log-regular ring and let $\fp$ be a height one prime ideal of $R$.
Then the following are equivalent.
\begin{enumerate}
\item
There exists a height one prime ideal $\fq$ of $\mathcal{Q}$ such that $\fp = \fq R$.
\item
 The intersection of $\im \alpha$ and $\alpha^{-1}(\fp)$ is not empty.
\end{enumerate}
\end{lemma}

\begin{proof}
The implication $(1) \Rightarrow (2)$ is obvious, hence let us consider the implication $(2) \Rightarrow (1)$.
Note that $\alpha^{-1}(\fp)$ is a height one prime ideal by assertion (2).
Since $\alpha$ is very solid and any element of $\mathcal{Q}$ does not map to $0$ by Lemma \ref{2241637}, we obtain $\alpha^{-1}(\fp)R = \fp$.
Hence assertion (1) holds.
\end{proof}



\begin{lemma}
Let $(R, \mathcal{Q}, \alpha)$ be a log ring and let $I, J$ be ideals of $\mcQ$.
Assume that $R$ is $\alpha$-flat.
Then
$$
\alpha(I)R \cap \alpha(J)R = \alpha(I \cap J)R
$$
holds.
\end{lemma}

\begin{proof}
Let us consider the following diagram:
\[
\begin{CD}
 0 @>>> \mathbb{Z}[I \cap J] @>>>    \mathbb{Z}[I] @>>>  \mathbb{Z}[I]/ \mathbb{Z}[I \cap J] @>>> 0 \\
     @.       @VVV    @VVV  @VVV  @. \\
 0 @>>> \mathbb{Z}[J] @>>> \mathbb{Z}[\mcQ] @>>> \mathbb{Z}[\mcQ]/\mathbb{Z}[J] @>>> 0.
\end{CD} 
\]
By the $\alpha$-flatness of $R$, we obtain the following diagram:
\[
\begin{CD}
 0 @>>> \mathbb{Z}[I \cap J]\otimes_{\mathbb{Z}[\mcQ]}R @>>>    \mathbb{Z}[I]\otimes_{\mathbb{Z}[\mcQ]}R @>>>  \mathbb{Z}[I]/ \mathbb{Z}[I \cap J]\otimes_{\mathbb{Z}[\mcQ]}R @>>> 0 \\
     @.       @VVV    @VVV  @VVV  @. \\
 0 @>>> \mathbb{Z}[J]\otimes_{\mathbb{Z}[\mcQ]}R @>>> \mathbb{Z}[\mcQ]\otimes_{\mathbb{Z}[\mcQ]}R @>>> \mathbb{Z}[\mcQ]/\mathbb{Z}[J]\otimes_{\mathbb{Z}[\mcQ]}R @>>> 0.
\end{CD} 
\]
This diagram is isomorphic to the following one:
\[
\begin{CD}
 0 @>>> \alpha(I \cap J)R @>>>    \alpha(I)R @>>>  \alpha(I)R/\alpha(I \cap J)R @>>> 0 \\
     @.       @VVV    @VVV  @VVV  @. \\
 0 @>>> \alpha(J)R @>>> R @>>> R/\alpha(J)R @>>> 0.
\end{CD}
\] 
Since the vertical arrows are injective, we obtain the following exact sequence by the snake lemma:
$$
0 \to \alpha(J)R/\alpha(I \cap J)R \to R/\alpha(I)R \xrightarrow{p} R/\alpha(I)R +\alpha(J)R \to 0.
$$
Thus since we obtain $\alpha(J)R/\alpha(I \cap J)R \cong \ker p = \alpha(I)R +\alpha(J)R/\alpha(I)R \cong \alpha(J)R/\alpha(I)R \cap \alpha(J)R$, the equality $\alpha(I \cap J)R  = \alpha(I)R \cap \alpha(J)R$ holds.
\end{proof}

\begin{lemma}\label{722913}
Let $(R, \mathcal{Q}, \alpha)$ be a log ring where $R$ is a domain and $\mcQ$ is cancellative.
Let $J, J'$ be a fractional ideal of $\mcQ$.
Then the equality $(J \cap J')R = JR \cap J'R$ holds.
\end{lemma}

\begin{proof}
Choose $x \in \mcQ$ such that $xJ, xJ' \subseteq \mcQ$.
Then it suffices to show that $x(JR \cap J'R) = xJR \cap xJ'R$, but this follows from $x(J \cap J') = xJ \cap xJ'$.
\end{proof}

\begin{lemma}\label{regfalphaflat}
    Let $(R,\mathcal{Q}, \alpha)$ be a local log-regular ring.
    Then $R$ is faithfully $\alpha$-flat.
\end{lemma}

\begin{proof}
    Let $M$ be a $\mathbb{Z}[\mathcal{Q}]$-module such that $M\otimes_{\mathbb{Z}[\mcQ]} R =0$.
    Then it suffices to show that $M$ is cyclic.
    In this case, $M$ is isomorphic to $\mathbb{Z}[\mcQ]/I$ where $I$ is the annihilator of some $x \in M$.
    Also, we have the equalities
    $$
    I = IR \cap \mathbb{Z}[\mcQ] = R \cap \mathbb{Z}[\mcQ] = \mathbb{Z}[\mcQ],
    $$
    where the first equality follows from the injectivity of $\mathbb{Z}[\mcQ] \to R$ and the second equality follows from the assumption $M \otimes_{\mathbb{Z}[\mcQ]} R=0$.
    This implies $M=0$, as desired.
\end{proof}

\begin{lemma}
Let $(R,\mathcal{Q},\alpha)$ be a local log-regular ring and let $I, J, J' \subseteq \mathcal{Q}^{\gp}$ be fractional ideals of $\mcQ$.
Assume that $I$ is finitely generated.
Then the following assertions hold.
\begin{enumerate}
\item
The equality $(J : I)R = (JR : IR)$ holds.
\item
$JR$ is equal to $J'R$ if and only if $J$ is equal to $J'$.
\item
$\Div(\alpha) : \Div(\mcQ) \to \Div(R)$ is well-defined and it is injective.
\end{enumerate}
\end{lemma}

\begin{proof}
We express $I = a_1\mcQ \cup \cdots \cup a_n\mcQ$ for some $a_1,\ldots, a_n \in \mcQ^{\gp}$.
Thus we obtain $(J : I) = a_1^{-1}J \cap \cdots \cap a_n^{-1}J$ and $(JR : IR) = a_1^{-1}JR \cap \cdots \cap a_n^{-1}JR$.
Here, by Lemma \ref{722913}, the equality $a_1^{-1}JR \cap \cdots \cap a_n^{-1}JR = (a_1^{-1}J \cap \cdots \cap a_n^{-1}J)R$ holds. Hence the assertion (1) holds.

Next to prove the assertion (2), we may assume that $J \subseteq J'$ after replacing $J'$ with $J' \cup J$.
Assume that the equality $JR = J'R$ holds.
Take an element $x \in \mcQ$ such that $xJ, xJ' \subseteq \mcQ$.
Then we have the short exact sequence
$$
\xymatrix{
xJ R \ar[r] & xJ' R \ar[r] & (\mathbb{Z}[xJ']/\mathbb{Z}[xJ]) \otimes_{\mathbb{Z}[\mcQ]} R \ar[r] & 0.
}
$$
Since we have $xJR = xJ'R$ and $R$ is faithfully $\alpha$-flat by Lemma \ref{regfalphaflat}, we obtain the equality $\mathbb{Z}[xJ] = \mathbb{Z}[xJ']$.
Hence $xJ = xJ'$ holds.
Since the converse implication is obvious, the assertion (2) holds.

Finally, the first assertion of (3) follows from (1), and the second follows from (2).
\end{proof}

\begin{proposition}\label{Clinjective}
Let $(R,\mathcal{Q}, \alpha)$ be a local log-regular ring.
Then $\Cl(\alpha) :\Cl(\mathcal{Q}) \to \Cl(R)$ is well-defined and it is injective.
\end{proposition}

\begin{proof}
This is the same as in \cite[Proposition 6.4.55]{GRBook}.
\end{proof}

\begin{lemma}\label{mainthm1}
Let $(R,\mathcal{Q},\alpha)$ be a complete local log-regular ring.
Let $S$ be the image of $\alpha$.
There exists an $R$-algebra $T$ such that $T$ is a regular local ring and $S^{-1}R \cong S^{-1}T$.
\end{lemma}

\begin{proof}
By replacing the monoid $\mcQ$ with $\mcQ \oplus \mathbb{N}^{\dim(R/I_\alpha)}$, we may assume $\dim(R/I_\alpha) = 0$.

First, suppose that $R$ is of equal characteristic.
Then $R$ is isomorphic to $k\llbracket \mathcal{Q}\rrbracket$ by Theorem \ref{Katostructure} (\ref{Katostructure1}).
Here, by Lemma \ref{transformation} (\ref{8159272}), the monoid homomorphism $\mathcal{Q} \hookrightarrow \mathbb{N}^r$ induces the injective ring homomorphism $k\llbracket \mathcal{Q} \rrbracket \hookrightarrow k\llbracket \mathbb{N}^r \rrbracket $.
Moreover, since $S^{-1}k\llbracket \mathcal{Q}\rrbracket $ is isomorphic to $S'^{-1}k\llbracket \mathbb{N}^r\rrbracket $ where $S'$ is the image of $S$, $S^{-1}k\llbracket \mathcal{Q} \rrbracket $ is a unique factorization domain.
Thus $k\llbracket \mathbb{N}^r\rrbracket$ is a desired regular local ring.

Next, suppose that $R$ is of mixed characteristic.
Then $R$ is isomorphic to $V\llbracket \mathcal{Q}\rrbracket /(\theta)$ by Theorem \ref{Katostructure} (\ref{Katostructure2}).
By the same discussion of the equal characteristic case, we obtain the injection $V\llbracket \mathcal{Q}\rrbracket /(\theta)V\llbracket \mathcal{Q}\rrbracket \hookrightarrow V\llbracket \mathbb{N}^r\rrbracket /(\theta)V\llbracket \mathbb{N}^r\rrbracket $, and $S^{-1}(V\llbracket \mathcal{Q}\rrbracket /(\theta)V\llbracket \mathcal{Q}\rrbracket )$ is isomorphic to $S'^{-1}(V\llbracket \mathbb{N}^r\rrbracket /(\theta)V\llbracket \mathbb{N}^r\rrbracket )$.
This also implies that $V\llbracket \mathbb{N}^r\rrbracket /(\theta)V\llbracket \mathbb{N}^r\rrbracket $ is a desired regular local ring.
\end{proof}

We prove the second main result in this paper.
For readers interested in the torsion part of the divisor class group of local log-regular rings, see \cite{INS22} and \cite{CLMST22}.

\begin{theorem}\label{fgcl}
Let $(R, \mathcal{Q}, \alpha)$ be a local log-regular ring.
Then $\Cl(\alpha) : \Cl(\mathcal{Q}) \to \Cl(R)$ is isomorphism.
In particular, the divisor class group $\Cl(R)$ is finitely generated.
\end{theorem}

\begin{proof}
Consider the composite map
$$
\Cl(\mcQ) \to \Cl(R) \to \Cl(\widehat{R}).
$$
Note that the former group homomorphism $\Cl(\mcQ) \to \Cl(R)$ is injective by Proposition \ref{Clinjective} and it is well-known that the latter group homomorphism $\Cl(R) \to \Cl(\widehat{R})$ is injective.
Since it suffices to show that $\Cl(\mcQ) \to \Cl(\widehat{R})$ is surjective, we may assume that $R$ is complete.

Let $S$ be as in Lemma \ref{mainthm1}.
By Nagate's theorem, we obtain the short exact sequence
$$
0 \to H \to \Cl(R) \to \Cl(S^{-1}R) \to 0,
$$
where $H$ is the subgroup of $\Cl(R)$ generated by classes of height one prime ideals that meets $S$.
Since $\Cl(S^{-1}R)$ is trivial by Lemma \ref{mainthm1}, we obtain $H=\Cl(R)$.
Moreover, we have an isomorphism $\Cl(\alpha) : \Cl(\mathcal{Q}) \xrightarrow{\cong} \im(\Cl(\alpha))=H$ by Lemma \ref{7231801}.
This implies that $\Cl(\alpha)$ is an isomorphism.
Finally, since the set of height one prime ideals of $\mathcal{Q}$ is finite by Lemma \ref{fundamentalmspc}, $\Cl(\mathcal{Q})$ is finitely generated.
Thus so is $\Cl(R)$.
\end{proof}

By combining Theorem \ref{fgcl} with Chouinard's Theorem \cite{Cho81} (or see \cite[Corollary 16.8]{GilBook84}), for a local log-regular ring $(R, \mcQ,\alpha)$, we obtain the isomorphisms
\begin{equation}\label{clisomlogmontoric}
   \Cl(\widehat{R}) \cong \Cl(R) \cong \Cl(\mcQ) \cong \Cl(k[\mcQ]).
\end{equation}

\begin{example}

Let $\mathcal{Q} \subseteq \mathbb{N}^4$ be the root closed submonoid generated by 
$$
x_1 := (1,1,0,0), x_2:= (0,0,1,1), x_3 := (1,0,0,1),\text{ and } x_4 := (0,1,1,0).
$$
Set $R:= W(k)\llbracket \mathcal{\mcQ} \rrbracket/(p-e^{x_4})W(k) \cong W(k)\llbracket x,y,z,w\rrbracket/(xy-zw,p-w) =W(k)\llbracket x,y,z\rrbracket/(xy-pz)$
where $k$ is a perfect field.
Then $(R, \mcQ, \alpha)$ is a local log-regular ring by Theorem \ref{Katostructure}, where $\alpha:\mathcal{Q} \to W(k)\llbracket \mcQ \rrbracket \to R$ is the composition of monoid homomorphisms.
Moreover, by applying the isomorphisms (\ref{clisomlogmontoric}), we obtain $\Cl(R) \cong \Cl(k[\mcQ]) = \mathbb{Z}$.
\end{example}

\begin{acknowledgement}
The author is extremely grateful to Ken-ichi Yoshida for his numerous suggestions regarding Section \ref{SectCanonical}.
The author especially thanks Alfred Geroldinger for his comments on Krull monoids and their divisor class groups and for introducing me to \cite{GFBook}.
The author is also deeply grateful to Bernd Ulrich for his advice on removing the completeness assumption in Theorem \ref{fgcl}.
The author thanks Shunsuke Takagi for the comments on Remark \ref{takagirem}.
The author thanks Kazuma Shimomoto for his helpful advice and kind support.
The author also thanks Ryo Ishizuka, Kei Nakazato, Kohsuke Shibata, Masataka Tomari, and Kei-ichi Watanabe for their several comments.
Finally, the author sincerely thanks the referee for their careful reading and many valuable suggestions.
\end{acknowledgement}

\bibliographystyle{alpha}
\bibliography{Ishiro_ref}

\end{document}